\documentclass{amsart}
\pdfoutput=1

\usepackage[utf8]{inputenc}
\usepackage{amsfonts,mathtools,amsmath,amsthm,amssymb,tikz,color,enumitem,latexsym,tikz-cd,mathrsfs,graphicx,verbatim,MnSymbol,tabularx}
\usepackage[algoruled,linesnumbered]{algorithm2e}
\usepackage[breaklinks]{hyperref}
\usepackage{thmtools}
\usepackage{thm-restate}
\usepackage[nameinlink,capitalise]{cleveref}
\usepackage{ogonek}
\usepackage[center]{caption}

\usepackage[backref=true]{biblatex}
\bibliography{sources.bib}
\allowdisplaybreaks

\hypersetup{
	colorlinks,
	hyperfootnotes=true,
	linkcolor={red!50!black},
	citecolor={blue!50!black},
	urlcolor={blue!80!black}
}

\DefineBibliographyStrings{english}{%
	backrefpage = {page},%
	backrefpages = {pages},%
}

\usetikzlibrary{arrows}

\tikzset{slopearrow/.style={sloped, anchor=south}}

\declaretheorem[name=Theorem,refname={Theorem,Theorems},
Refname={Theorem,Theorems},numberwithin=section]{thm}

\declaretheorem[name=Proposition,refname={Proposition,Propositions},
Refname={Proposition,Propositions},sibling=thm]{prop}
\declaretheorem[name=Corollary,refname={Corollary,Corollaries},
Refname={Corollary,Corollaries},sibling=thm]{cor}
\declaretheorem[name=Lemma,refname={Lemma,Lemmas},
Refname={Lemma,Lemmas},sibling=thm]{lem}

\declaretheorem[name=Remark,refname={Remark,Remarks},
Refname={Remark,Remarks},style=remark,sibling=thm]{rmk}

\declaretheorem[name=Example,refname={Example,Examples},
Refname={Example,Examples},style=remark,sibling=thm]{ex}

\def\CC{\mathbb C}

\def\RR{\mathbb R}
\def\ZZ{\mathbb Z}
\def\Z{\mathbb Z}
\def\QQ{\mathbb Q}
\def\Q{\mathbb Q}

\def\F{\mathbb F}
\def\PP{\mathbb P}
\def\bbA{\mathbb A}

\def\k{k}

\DeclareMathOperator\GL{GL}

\DeclareMathOperator\GO{O}

\DeclareMathOperator\SO{SO}

\DeclareMathOperator\Spec{Spec}

\def\base{{\GL_2}}
\DeclareMathOperator\ord{ord}

\title[Rational configuration problems]{Rational configuration problems and a family of curves}
\author{Jonathan Love}
\address[Jonathan Love]{McGill University}
\email{jon.love@mcgill.ca}
\thanks{Supported by CRM-ISM postdoctoral fellowship}
\date{November 2021}

\begin{document}
	
	\begin{abstract}
		Given $\eta=\begin{psmallmatrix}
			a&b\\c&d
		\end{psmallmatrix}\in \GL_2(\QQ)$, we consider the number of rational points on the genus one curve
		\[H_\eta:y^2=(a(1-x^2)+b(2x))^2+(c(1-x^2)+d(2x))^2.\]
		We prove that the set of $\eta$ for which $H_\eta(\QQ)\neq\emptyset$ has density zero, and that if a rational point $(x_0,y_0)\in H_\eta(\QQ)$ exists, then $H_\eta(\QQ)$ is infinite unless a certain explicit polynomial in $a,b,c,d,x_0,y_0$ vanishes.
		
		Curves of the form $H_\eta$ naturally occur in the study of configurations of points in $\RR^n$ with rational distances between them. As one example demonstrating this framework, we prove that if a line through the origin in $\RR^2$ passes through a rational point on the unit circle, then it contains a dense set of points $P$ such that the distances from $P$ to each of the three points $(0,0)$, $(0,1)$, and $(1,1)$ are all rational. We also prove some results regarding whether a rational number can be expressed as a sum or product of slopes of rational right triangles.
	\end{abstract}
	
	\maketitle

\section{Introduction}\label{sec:intro}

\subsection{A family of curves}

Fix $\eta:=\begin{psmallmatrix}
	a&b\\c&d
\end{psmallmatrix}\in \GL_2(\QQ)$, and let $H_\eta$ be the curve defined by
\begin{align}\label{eq:Heta}
	H_\eta:y^2=(a(z^2-x^2)+b(2xz))^2+(c(z^2-x^2)+d(2xz))^2
\end{align}
in the weighted projective plane where $x,y,z$ have degree $1,2,1$, respectively. Rational points on this curve correspond to vectors $\begin{psmallmatrix}
	u\\v
\end{psmallmatrix}\in\QQ^2\setminus\{0\}$ such that both $\begin{psmallmatrix}
u\\v
\end{psmallmatrix}$ and $\eta\begin{psmallmatrix}
u\\v
\end{psmallmatrix}$ have rational length, and as a result, curves of this form can be used to describe solutions to a collection of rational configuration problems; see \cref{sec:ratconfigs} for more details. In this paper we study the loci of points $\eta$ for which $H_\eta$ has zero, finitely many, or infinitely many rational points.

First, we show that for most values of $\eta$, the curve $H_\eta$ has no rational points. 
\begin{thm}\label{thm:upperlower}
	Let $\mathcal{L}(X)$ be the set of $\eta\in\GL_2(\QQ)$ with $a,b,c,d\in \ZZ\cap [-X,X]$ such that $H_\eta(\QQ_v)$ is nonempty for all $v\in\{\infty, 2,3,5,7,\ldots\}$. Then for some constant $C>0$,
	\[\frac{|\mathcal{L}(X)|}{(2X)^4}<C(\log X)^{-1/4}.\]
\end{thm}
The proof is given in \cref{sec:upperbound}. Note that $H_\eta\simeq H_{m\eta}$ for any positive integer $m$, so by clearing denominators, every $H_\eta$ is isomorphic to one of the curves counted in \cref{thm:upperlower}. For the sake of comparison, consider the following result by Bhargava, Cremona, and Fisher.
\begin{thm}[{\cite[Theorem 3]{bcf}}]\label{bigfamilycompare}
	Let $\mathcal{L}'(X)$ denote the set of $(a,b,c,d,e)\in(\ZZ\cap[-X,X])^5$ such that 
	\[y^2=ax^4+bx^3+cx^2+dx+e\]
	has a $\QQ_v$ point for all $v\in\{\infty, 2,3,5,7,\ldots\}$. Then
	\[\lim_{X\to\infty} \frac{|\mathcal{L}'(X)|}{(2X)^5}\approx 0.7596.\]
\end{thm}
We see that the subfamily $H_\eta$ differs from the larger family, in that far fewer specializations are everywhere locally soluble.

Now suppose we restrict our attention to the collection of points $\eta$ for which $H_\eta$ does contain a rational point. In this case we have a stronger classification. Let $\eta^t$ denote the transpose of $\eta$.
\begin{prop}\label{prop:isomorphicelliptic}
	Suppose $H_\eta$ (\cref{eq:Heta}) has a rational point. If $\eta\eta^t$ is a scalar matrix, then $\det\eta=\lambda^2$ for some $\lambda\in\QQ^\times$ and $H_\eta$ is a union of two rational conics,
	\[y=\pm\lambda(x^2+1).\] 
	Otherwise $H_\eta$ is isomorphic to 
	\[E_{r,s}:y^2=x^3+(1+r^2+s^2)x^2+s^2x\]
	for some $r,s\in\QQ$ with $s\neq 0$ and $(r,s)\neq (0,\pm1)$. 
\end{prop}
A proof is given in \cref{sec:singular} using the fact that the isomorphism type of $H_\eta$ is invariant under acting on the left and right of $\eta$ by elements of the orthogonal group $\GO_2(\QQ)$. An explicit change of variables expressing $r,s$ in terms of $a,b,c,d$ and the rational point $(x_0:y_0:z_0)\in H_\eta(\QQ)$ is given by \cref{lem:solubletriangular}.

\begin{restatable}{thm}{Qinfpointsrestate}\label{thm:Qinfpoints}
	Let $r,s\in\QQ$ with $s\neq 0$ and $(r,s)\neq (0,\pm1)$. The point $(-1,r)\in E_{r,s}(\QQ)$ is non-torsion if and only if $r\neq 0$, $s\neq \pm1$, and $4r^2s\neq \pm(1-s^2)^2$.
\end{restatable}
In particular, for most of the values $\eta$ such that $H_\eta(\QQ)$ is nonempty, $H_\eta(\QQ)$ is actually infinite. The proof of this result is given in \cref{sec:nonsingular}. 
We discuss several applications of this result to rational distance problems in \cref{sec:ratconfigs}, but mention one here as a representative example.

\begin{restatable}{cor}{squaredistrestate}
	\label{thm:squaredist}
	On any line of the form $x=0$ or $y=\frac{2t}{1-t^2}x$ for $t\in\QQ\setminus\{-1,0,1\}$, there exists a dense set of points with rational distance from each of $(0,0)$, $(0,1)$, and $(1,1)$.
\end{restatable}

In fact we prove a stronger result: there is an infinite collection of curves $C_n$ in the plane such that the intersection points of the curves $C_n$ with any fixed line $y=\frac{2t}{1-t^2}x$ (for $t\in\QQ\setminus\{-1,0,1\}$) gives a dense set of solutions to the three-distance problem within the given line (\cref{cor:squaredist_param}). 

Even in the cases where $(-1,r)\in E_{r,s}(\QQ)$ is torsion, there are still several cases in which we can prove $E_{r,s}(\QQ)$ has positive rank. We discuss these in more depth in \cref{sec:sumprod}, but note one special case here. Let
\begin{align*}
	\mathcal{S}'&=\{\alpha\in\QQ:\sqrt{\alpha^2+1}\in\QQ\}
\end{align*}
denote the set of slopes of rational right triangles (including negatives and zero).
\begin{prop}\label{prop:threesum}
	For all $t\in\QQ$, the equations $x_1+x_2+x_3=t$ and $x_1x_2x_3=t$ each have an infinite set of solutions with $x_1,x_2,x_3\in \mathcal{S}'$.
\end{prop}

See \cref{sec:sumprod} for a proof.

\subsection{Rational configuration problems}\label{sec:ratconfigs}

Given a finite simple graph $G=(V,E)$, an embedding $\phi:V\hookrightarrow\RR^n$ is a \emph{rational configuration} if the distance $d(\phi(v),\phi(w))$ is rational for all $(v,w)\in E$. We may add some additional constraints to the set of allowable embeddings (for instance, we may require some pairs of edges to be the same length, or to meet at right angles), and in doing so we obtain a corresponding \emph{rational configuration problem}: to determine whether there exists a rational configuration satisfying the desired constraints, and if so, to classify or count the number of rational configurations. We describe a list of sample rational configurations below; the corresponding graphs can be found in \cref{ratdistprobs}.
\begin{itemize}
	\item ``Adjacent rectangles:'' Find two rectangles sharing an edge such that the distance between any two vertices is rational.
	\item ``Detour:'' Fix parameters $r,s,t\in\QQ^\times$. Find a point $x$ such that $(x,0)$ has rational distance to $(0,0)$, $(r,0)$, $(0,s)$, and $(r,t)$. (A traveller is going from $(0,s)$ to $(r,t)$, but has to take a detour to stop at the $x$-axis along the way; can they do so using only two straight paths of rational length?)
	\item ``Perfect cuboid:'' Find a rectangular prism such that the distance between any two vertices is rational.
	\item ``Body cuboid:'' Find a rectangular prism such that the distance between any two vertices that share a face are rational.
	\item ``Square four-distance:'' Find a point $(x,y)\in\mathbb{R}^2$ such that the distance to each of $(0,0)$, $(1,0)$, $(0,1)$, and $(1,1)$ is rational.
	\item ``Square three-distance:'' Find a point $(x,y)\in\mathbb{R}^2$ such that the distance to each of $(0,0)$, $(0,1)$, and $(1,1)$ is rational.
	\item ``Rectangle four-distance:'' Find $r\in\mathbb{Q}^\times$ and a point $(x,y)\in\mathbb{R}^2$ such that the distance to each of $(0,0)$, $(0,1)$, $(r,0)$, and $(r,1)$ is rational.
	\item ``Rational distances under M\"{o}bius transformation:'' Fix $a,b,c,d\in\QQ$ with $ad-bc\neq 0$. Find $z\in\CC$ such that $z$ and $\frac{az+b}{cz+d}$ both have rational distance from $0$.
\end{itemize}

\begin{table}
	{\footnotesize
		\begin{tabularx}{\textwidth}{>{\raggedright}X>{\raggedright}X>{\raggedright}X>{\raggedright\arraybackslash}X}
			Configuration & Graph & Equation & Solutions given by \\
			\hline
			Adjacent rectangles &  
			\begin{tikzpicture}
				\filldraw[black] (0,0) node (v3) {} circle (1.5pt);
				\filldraw[black] (0.7,0) node (v4) {} circle (1.5pt);
				\filldraw[black] (1.8,0) node (v5) {} circle (1.5pt);
				\filldraw[black] (0,0.9) node (v1) {} circle (1.5pt);
				\filldraw[black] (0.7,0.9) node (v2) {} circle (1.5pt);
				\filldraw[black] (1.8,0.9) node (v6) {} circle (1.5pt);
				\node at (0.35,-0.2) {$\alpha_1$};
				\node at (1.25,-0.2) {$\alpha_2$};
				\node at (-0.2,0.45) {$1$};
				\draw  (v1) edge (v2);
				\draw  (v1) edge (v3);
				\draw  (v3) edge (v4);
				\draw  (v4) edge (v5);
				\draw  (v5) edge (v6);
				\draw  (v6) edge (v2);
				\draw  (v2) edge (v4);
				\draw  (v1) edge (v4);
				\draw  (v2) edge (v4);
				\draw  (v2) edge (v5);
				\draw  (v3) edge (v6);
				\node (v7) at (0,1.1) {};
				\node (v8) at (1.8,1.1) {};
			\end{tikzpicture} & $\alpha_1+\alpha_2=\alpha_3$ & $E_{\alpha_3,1}(\QQ)$ \newline ($\infty$ for all $\alpha_3\in\mathcal{S}$: \newline \cref{prop:pythslopesum})\\
			Detour ($r,s\in\QQ^\times$) &  
			\begin{tikzpicture}
				\filldraw[black] (0,0) node (v3) {} circle (1.5pt);
				\filldraw[black] (0.9,0) node (v4) {} circle (1.5pt);
				\filldraw[black] (2,0) node (v5) {} circle (1.5pt);
				\filldraw[black] (0,0.75) node (v1) {} circle (1.5pt);
				\filldraw[black]  node (v2) {} circle (1.5pt);
				\filldraw[black] (2,1) node (v6) {} circle (1.5pt);
				\node at (0.45,-0.2) {$s\alpha_1$};
				\node at (1.5,-0.2) {$\alpha_2$};
				\node (v9) at (-0.2,0.45) {$s$};
				\node (v9) at (2.25,0.5) {$1$};
				\node (v9) at (1,-0.55) {$r$};

				\draw[dashed]  (0,-0.35) edge (2,-0.35);
				\draw  (v1) edge (v4);
				\draw  (v4) edge (v6);
				\draw  (v1) edge (v3);
				\draw  (v3) edge (v4);
				\draw  (v4) edge (v5);
				\draw  (v5) edge (v6);
			\end{tikzpicture} & $s\alpha_1+\alpha_2=r$ & $E_{r,s}(\QQ)$ \newline ($\infty$ if $|s|\neq 1$ and $4r^2s\neq \pm(1-s^2)^2$: \newline \cref{thm:Qinfpoints})\\
			Perfect cuboid &  \begin{tikzpicture}
				\filldraw[black] (0,0) node (v3) {} circle (1.5pt);
				\filldraw[black] (1.5,0) node (v4) {} circle (1.5pt);
				\filldraw[black] (2,1.3) node (v5) {} circle (1.5pt);
				\filldraw[black] (0,0.8) node (v1) {} circle (1.5pt);
				\filldraw[black] (1.5,0.8) node (v2) {} circle (1.5pt);
				\filldraw[black] (0.5,1.3) node (v6) {} circle (1.5pt);
				\node at (0.75,-0.2) {$\alpha_1$};
				\node at (1.9,0.1) {$\alpha_2$};
				\node at (-0.2,0.5) {$1$};
				\filldraw[black] (0.5,0.5) node (v7) {} circle (1.5pt);
				\filldraw[black] (2,0.5) node (v8) {} circle (1.5pt);
				\draw  (v1) edge (v2);
				\draw  (v2) edge (v4);
				\draw  (v5) edge (v2);
				\draw  (v4) edge (v3);
				\draw  (v4) edge (v8);
				\draw  (v8) edge (v5);
				\draw  (v5) edge (v6);
				\draw  (v6) edge (v7);
				\draw  (v7) edge (v3);
				\draw  (v3) edge (v1);
				\draw  (v1) edge (v6);
				\draw  (v7) edge (v8);
				\draw  (v2) edge (v6);
				\draw  (v5) edge (v4);
				\draw  (v2) edge (v3);
				\draw  (v4) edge (v6);
			\end{tikzpicture} & $\alpha_1^2+\alpha_2^2=\alpha_3^2$ & Unknown \\
			Body cuboid & \begin{tikzpicture}
				\filldraw[black] (0,0) node (v3) {} circle (1.5pt);
				\filldraw[black] (1.5,0) node (v4) {} circle (1.5pt);
				\filldraw[black] (2,1.3) node (v5) {} circle (1.5pt);
				\filldraw[black] (0,0.8) node (v1) {} circle (1.5pt);
				\filldraw[black] (1.5,0.8) node (v2) {} circle (1.5pt);
				\filldraw[black] (0.5,1.3) node (v6) {} circle (1.5pt);
				\node at (0.75,-0.2) {$\alpha_1$};
				\node at (1.9,0.1) {$\alpha_2$};
				\node at (-0.2,0.5) {$1$};
				\filldraw[black] (0.5,0.5) node (v7) {} circle (1.5pt);
				\filldraw[black] (2,0.5) node (v8) {} circle (1.5pt);
				\draw  (v1) edge (v2);
				\draw  (v2) edge (v4);
				\draw  (v5) edge (v2);
				\draw  (v4) edge (v3);
				\draw  (v4) edge (v8);
				\draw  (v8) edge (v5);
				\draw  (v5) edge (v6);
				\draw  (v6) edge (v7);
				\draw  (v7) edge (v3);
				\draw  (v3) edge (v1);
				\draw  (v1) edge (v6);
				\draw  (v7) edge (v8);
				\draw  (v2) edge (v6);
				\draw  (v5) edge (v4);
				\draw  (v2) edge (v3);
			\end{tikzpicture} & $\alpha_1\alpha_3=\alpha_2$ & $E_{0,\alpha_3}(\QQ)$ \\
			(Square) Four-distance &  \begin{tikzpicture}
				\filldraw[black] (0,0) node (v3) {} circle (1.5pt);
				\filldraw[black] (0.45,0.95) node (v4) {} circle (1.5pt);
				\filldraw[black] (1.5,1.5) node (v5) {} circle (1.5pt);
				\filldraw[black] (0,1.5) node (v1) {} circle (1.5pt);
				\filldraw[black] (1.5,0) node (v2) {} circle (1.5pt);
				\filldraw[black] (0,0.95) node (v9) {} circle (1.5pt);
				\filldraw[black] (1.5,0.95) node (v6) {} circle (1.5pt);
				\filldraw[black] (0.45,0) node (v7) {} circle (1.5pt);
				\filldraw[black] (0.45,1.5) node (v8) {} circle (1.5pt);
				
				\draw  (v1) edge (v9);
				\draw  (v9) edge (v3);
				\draw  (v3) edge (v7);
				\draw  (v7) edge (v2);
				\draw  (v2) edge (v6);
				\draw  (v6) edge (v5);
				\draw  (v5) edge (v8);
				\draw  (v8) edge (v1);
				\draw  (v9) edge (v4);
				\draw  (v4) edge (v6);
				\draw  (v8) edge (v4);
				\draw  (v4) edge (v7);
				\draw  (v4) edge (v3);
				\draw  (v4) edge (v1);
				\draw  (v4) edge (v5);
				\draw  (v4) edge (v2);
				\node at (0.25,1.65) {$1$};
				\node at (1,1.65) {$\alpha\alpha_2$};
				\node at (-0.2,0.5) {$\alpha_3$};
				\node at (-0.2,1.25) {$\alpha_1$};
			\end{tikzpicture}  & $\alpha_1\alpha_2=\alpha_3\alpha_4=\alpha_1+\alpha_3-1$ & Unknown \\
			(Square) three-distance & \begin{tikzpicture}
				\filldraw[black] (0,0) node (v3) {} circle (1.5pt);
				\filldraw[black] (0.45,0.95) node (v4) {} circle (1.5pt);
				\filldraw[black] (1.5,1.5) node (v5) {} circle (1.5pt);
				\filldraw[black] (0,1.5) node (v1) {} circle (1.5pt);
				\filldraw[black] (1.5,0) node (v2) {} circle (1.5pt);
				\filldraw[black] (0,0.95) node (v9) {} circle (1.5pt);
				\filldraw[black] (1.5,0.95) node (v6) {} circle (1.5pt);
				\filldraw[black] (0.45,0) node (v7) {} circle (1.5pt);
				\filldraw[black] (0.45,1.5) node (v8) {} circle (1.5pt);
				
				\draw  (v1) edge (v9);
				\draw  (v9) edge (v3);
				\draw  (v3) edge (v7);
				\draw  (v7) edge (v2);
				\draw  (v2) edge (v6);
				\draw  (v6) edge (v5);
				\draw  (v5) edge (v8);
				\draw  (v8) edge (v1);
				\draw  (v9) edge (v4);
				\draw  (v4) edge (v6);
				\draw  (v8) edge (v4);
				\draw  (v4) edge (v7);
				\draw  (v4) edge (v3);
				\draw  (v4) edge (v1);
				\draw  (v4) edge (v5);
				\node at (0.25,1.65) {$1$};
				\node at (1,1.65) {$\alpha\alpha_2$};
				\node at (-0.2,0.5) {$\alpha_3$};
				\node at (-0.2,1.25) {$\alpha_1$};
			\end{tikzpicture} & $\alpha_1\alpha_2=\alpha_1+\alpha_3-1$ & $E_{-1,1-\alpha_3}(\QQ)$ \newline ($\infty$ for all $\alpha_3\in\mathcal{S}$: \newline \cref{thm:squaredist}) \\
			(Rectangle) four-distance & \begin{tikzpicture}
				\filldraw[black] (0,0) node (v3) {} circle (1.5pt);
				\filldraw[black] (0.45,0.7) node (v4) {} circle (1.5pt);
				\filldraw[black] (1.5,1.2) node (v5) {} circle (1.5pt);
				\filldraw[black] (0,1.2) node (v1) {} circle (1.5pt);
				\filldraw[black] (1.5,0) node (v2) {} circle (1.5pt);
				\filldraw[black] (0,0.7) node (v9) {} circle (1.5pt);
				\filldraw[black] (1.5,0.7) node (v6) {} circle (1.5pt);
				\filldraw[black] (0.45,0) node (v7) {} circle (1.5pt);
				\filldraw[black] (0.45,1.2) node (v8) {} circle (1.5pt);
				
				\draw  (v1) edge (v9);
				\draw  (v9) edge (v3);
				\draw  (v3) edge (v7);
				\draw  (v7) edge (v2);
				\draw  (v2) edge (v6);
				\draw  (v6) edge (v5);
				\draw  (v5) edge (v8);
				\draw  (v8) edge (v1);
				\draw  (v9) edge (v4);
				\draw  (v4) edge (v6);
				\draw  (v8) edge (v4);
				\draw  (v4) edge (v7);
				\draw  (v4) edge (v3);
				\draw  (v4) edge (v1);
				\draw  (v4) edge (v2);
				\draw  (v4) edge (v5);
				\node at (0.25,1.35) {$1$};
				\node at (1,1.35) {$\alpha_1\alpha_2$};
				\node at (-0.2,0.35) {$\alpha_1$};
				\node at (-0.2,0.95) {$\alpha_3$};
			\end{tikzpicture} & $\alpha_1\alpha_2=\alpha_3\alpha_4$ & $E_{0,\alpha_3\alpha_4}$ \\
			Rational distances under M\"{o}bius transformation, $\eta=\begin{psmallmatrix}
				a&b\\c&d
			\end{psmallmatrix}\in \GL_2(\QQ)$ &  \begin{tikzpicture}
				\filldraw[black] (0,0) node (v3) {} circle (1.5pt);
				\filldraw[black] (1.5,0) node (v4) {} circle (1.5pt);
				\filldraw[black] (0.5,1.25) node (v5) {} circle (1.5pt);
				\filldraw[black] (1.5,0.8) node (v2) {} circle (1.5pt);
				\filldraw[black] (0,1.25) node (v6) {} circle (1.5pt);
				
				\draw  (v3) edge (v4);
				\draw  (v4) edge (v2);
				\draw  (v2) edge (v3);
				\draw  (v3) edge (v6);
				\draw  (v6) edge (v5);
				\draw  (v5) edge (v3);
				\node at (-0.55,0.7) {$c\alpha_1+d$};
				\node at (0.25,1.5) {$a\alpha_1+b$};
				\node at (0.75,-0.2) {$\alpha_1$};
				\node at (1.65,0.4) {$1$};
			\end{tikzpicture} & $(a\alpha_1+b)\alpha_2=(c\alpha_1+d)$ & $H_{\eta}(\QQ)$
	\end{tabularx}}
	\caption{Diagrams of rational distance problems. Rational configurations are given by solutions to the given equation with $\alpha_i\in\mathcal{S}$ for all $i$. By \cref{prop:param}, these configurations are parametrized by rational points on a curve $H_\eta$. If a distinguished rational point on $H_\eta$ is known, then the isomorphic elliptic curve $E_{r,s}$ is given (\cref{prop:isomorphicelliptic}), and labeled with $\infty$ in the cases that $E_{r,s}(\QQ)$ is known to be infinite.} \label{ratdistprobs}
\end{table}

The perfect cuboid problem and square four-distance problem are classic unsolved problems (see \cref{sec:relatedprobs}); this paper does not present a solution to either of them. However, we can put all the remaining problems in this list into a common framework. Define
\begin{align*}
	\mathcal{S}&=\{(u:v)\in\mathbb{P}^1(\QQ)\mid\sqrt{u^2+v^2}\in\QQ\},
\end{align*}
so that whenever $u,v$, not both zero, are the legs of a (possibly degenerate) rational right triangle, the slope of the triangle is in $\mathcal{S}$. Then for distinct $P_1,P_2\in\mathbb{Q}^2$, the distance between $P_1$ and $P_2$ is rational if and only if the line between $P_1$ and $P_2$ has slope in $\mathcal{S}$. Using this observation, we can parametrize solutions to many rational configuration problems by finding elements of $\mathcal{S}$ satisfying simple polynomial relations.
\begin{ex}
	Given a hypothetical solution to the perfect cuboid problem, we can scale the solution so that one edge length has length $1$; this implies there exist $\alpha_1,\alpha_2\in\Q\setminus\{0\}$ such that 
	\[1+\alpha_1^2, \quad 1+\alpha_2^2,\quad \alpha_1^2+\alpha_2^2, \quad\text{and}\quad 1+\alpha_1^2+\alpha_2^2\]
	are all perfect squares. If we set $\alpha_3=\sqrt{\alpha_1^2+\alpha_2^2}$, then the polynomial constraints above are equivalent to requiring
	\[\alpha_1^2+\alpha_2^2=\alpha_3^2\qquad\text{for some }\alpha_1,\alpha_2,\alpha_3\in\mathcal{S}\setminus\{(1:0),(0:1)\}.\]
\end{ex}
Similar polynomial constraints for each of the problems above are listed in \cref{ratdistprobs}. Note that for every problem in \cref{ratdistprobs} besides the perfect cuboid problem and the square four-distance problem, rational configurations correspond to solutions in $\mathcal{S}$ to a single polynomial in multiple variables that is linear in each variable.
\begin{prop}\label{prop:param}
	Let $\eta=\begin{psmallmatrix}
		a&b\\c&d
	\end{psmallmatrix}\in\GL_2(\QQ)$, and let $F_\eta$ be the curve in $\mathbb{P}^1\times \mathbb{P}^1$ defined by
	\[F_\eta:ax_1x_2+bx_1z_2+cz_1x_2+dz_1z_2=0.\]
	There is a degree $4$ morphism $\Phi:H_\eta\to F_\eta$ inducing a surjection 
	\[H_\eta(\QQ)\to F_\eta(\QQ)\cap (\mathcal{S}\times\mathcal{S}).\]
\end{prop}
This follows from \cref{prop:bijection}. \cref{prop:param} shows that for a wide collection of problems, rational configurations can be classified using rational points on curves of the form $H_\eta$. We can use this observation to show that some rational configuration problems have infinitely many rational configurations. In some cases, such as the detour problem and the square three-distance problem, the infinitude of solutions will be a consequence of \cref{thm:Qinfpoints}. For others, including the adjacent rectangles, body cuboid, and rectangle four-distance problems, the corresponding curve $E_{r,s}$ lands in one of the exceptional cases of \cref{thm:Qinfpoints}, and so we cannot immediately conclude that there are infinitely many solutions.

\subsection{Outline}

We begin with a discussion of some related problems and their histories in \cref{sec:relatedprobs}. In \cref{sec:setup} we analyze the algebraic structure of the family $H_\eta$, in particular showing that the isomorphism type of $H_\eta$ is invariant under a left- and right-action of the orthogonal group (\cref{sec:doubcos}). We then analyze the singular fibers (\cref{sec:singular}), followed by the non-singular fibers that contain a rational point (\cref{sec:nonsingular}), proving that most fibers of this type have infinitely many rational points (\cref{thm:Qinfpoints}). Completing our study of rational points on the fibers, \cref{sec:upperbound} contains a proof the set of fibers containing a rational point has low density (\cref{thm:upperlower}). Note that \cref{sec:upperbound} only requires \cref{sec:assumptions} and \cref{sec:Hdef} from \cref{sec:setup}.

We conclude with some applications of these results in \cref{sec:applications}, focusing primarily on the square three-distance problem.

\subsection{Acknowledgments}

The author was supported by a CRM-ISM Postdoctoral Fellowship during the writing of this article, and would like to thank Andrew Granville, Sun-Kai Leung, Michael Lipnowski, Henri Darmon, Eyal Goren, Allysa Lumley, Olivier Mila, Wanlin Li, and Valeriya Kovaleva for helpful discussions.

\section{Prior work on related problems}\label{sec:relatedprobs}

There are a number of open problems regarding rational configurations; in this section we will focus on two of them, namely the perfect cuboid problem in \cref{sec:cuboid} and the square four-distance problem in \cref{sec:squarevertex} (both of these are discussed at greater length in~\cite{guy}). In each case, we show that the problem is equivalent to the existence of a Pythagorean solution of a certain polynomial or system of polynomials. Finally, in \cref{sec:cong}, we compare to the congruent number problem.

\subsection{Perfect cuboid problem}\label{sec:cuboid}

While the perfect cuboid problem is open, significant progress has been made towards studying the ``body cuboid'' problem, which is to give a cuboid in which all edges and all face diagonals (but not necessarily the body diagonal) have rational lengths. If $1,\alpha_1,\alpha_2$ are the edge lengths of a body cuboid, then $\alpha_1^2+1$, $\alpha_2^2+1$, and $\alpha_1^2+\alpha_2^2$ are all perfect squares; the first two conditions say that $\alpha_1,\alpha_2\in\mathcal{S}$ and the third is equivalent to requiring $\frac{\alpha_2}{\alpha_1}\in\mathcal{S}$.

For each fixed $\alpha_3\in\mathcal{S}$, the values $\alpha_1,\alpha_2\in\mathcal{S}$ satisfying $\alpha_1\alpha_3=\alpha_2$ are parametrized by an elliptic curve (\cref{prop:param} and \cref{prop:isomorphicelliptic}). This association between body cuboids and a family of elliptic curves is well-studied; Luijk has an in-depth survey~\cite{luijk} that mentions this association as well as many other known results about perfect cuboids. Halbeisen and Hungerb\"uler~\cite{halbeisen} investigate this problem as well. Given a fixed $\alpha_3=\frac ba$, they associate solutions $\alpha_1,\alpha_2\in\mathcal{S}\setminus\{0\}$ satisfying $\alpha_1\alpha_3=\alpha_2$ to rational points on the elliptic curve
\begin{align}
	E:y^2=x^3+(a^2+b^2)x^2+a^2b^2x.
\end{align}
\cref{prop:param} and \cref{prop:isomorphicelliptic} recovers this classification. They show that there is a subgroup of $E(\QQ)$ isomorphic to $\ZZ/2\ZZ\times \ZZ/4\ZZ$ which give degenerate solutions to the corresponding rational distance problem. Ruling out other possible torsion points, they conclude~\cite[Theorem 8]{halbeisen} that nondegnerate solutions exist if and only if $E(\QQ)$ has positive rank. In this case they call $(a,b)$ a \emph{double-pythapotent pair}. 

\subsection{Four-distance problem}\label{sec:squarevertex}

As with the perfect cuboid problem, the four-distance problem is currently out of reach, but a slightly weaker variant has many known solutions. The three-distance problem is to find points $P=(x,y)\in\mathbb{R}^2$ with rational distance to $(0,0)$, $(0,1)$, and $(1,1)$. The coordinates $x,y$ are not a priori assumed to be rational, but since $x^2+y^2$, $x^2+(1-y)^2$, and $(1-x)^2+(1-y)^2$ must all be rational, the differences $2y-1$ and $2x-1$ must also be rational, so in fact $P\in\QQ^2$. We can then scale by an element of $\QQ^\times$ so that $x=1$, and a solution to the square three-distance problem is equivalent to the existence of $\alpha_1,\alpha_2,\alpha_3\in\mathcal{S}$ satsifying $1+\alpha_1\alpha_2=\alpha_1+\alpha_3$.

For many years it was believed that there were no solutions to the three-distance problem aside from points on the coordinate axes. The first one-parameter family of nontrivial solutions was found in 1967 by J.H. Hunter, and then many more infinite families were found in rapid succession; a historical overview is given by Berry, who also presents an ``extraordinary abundance'' of solutions lying in infinitely many one-parameter families~\cite{berry}. We observe that the families of solutions obtained in \cref{thm:squaredist} are distinct from those that appear in \cite[Table 4]{berry}, though it is unclear whether any (or all) of the one-parameter families we consider are eventually accounted for by Berry's construction.

\subsection{Congruent number problem}\label{sec:cong}

A rational number $n\in\QQ$ is a \emph{congruent number} if it is the area of a right triangle with rational edge lengths; that is, if there is a solution to
\begin{align}\label{eq:cong}
	a^2+b^2=c^2\qquad\text{and}\qquad\frac12 ab=n,\qquad a,b,c\in\QQ^\times.
\end{align}
The ``congruent number problem'' is to determine whether a given $n\in\QQ$ is a congruent number.
This problem is not a rational configuration problem, but the underlying methods used to study these two problems are similar enough that a comparison is worthwhile.

There is a well-known approach to studying the congruent number problem; see for example the expositions~\cite{kconrad} and~\cite{chandrasekar}. For fixed $n$, any solution to \cref{eq:cong} corresponds to a rational point on an elliptic curve over $\QQ$ defined by
\begin{align}\label{eq:congcurve}
E^{(n)}:y^2=x^3-n^2x.
\end{align}
There are ``degenerate points'' in $E^{(n)}(\QQ)$ that do not correspond to solutions; it can be shown that the set of degenerate points equals the torsion subgroup of $E^{(n)}(\QQ)$. Thus $n$ is a congruent number if and only if $E^{(n)}(\QQ)$ has positive rank. A formula due to Tunnell can be used to determine whether the analytic rank of $E^{(n)}$ is zero or positive~\cite{tunnell}, so by assuming the Birch and Swinnerton-Dyer conjecture, this gives a criterion that determines whether a given number is congruent.

Many aspects of this paper are modeled off of the approach described for studying the congruent number problem. To put the two problems on a common footing, note that $n$ is a congruent number if and only if $x_1=a^2$ and $x_2=\frac ba$ give a solution to
\begin{align}\label{eq:congnumeq}
	x_1x_2-2n=0,\qquad x_1\in(\QQ^\times)^2,\;x_2\in\mathcal{S}.
\end{align}
Both $\mathcal{S}$ and $(\QQ^\times)^2$ can be represented as the image of $\bbA^1(\QQ)$ under the image of a degree $2$ rational map $\bbA^1\to\bbA^1$. The curve $E^{(n)}$ comes equipped with a degree $4$ rational map to the variety defined by $x_1x_2-2n=0$, and non-degenerate points in $E^{(n)}(\QQ)$ map to solutions to \cref{eq:congnumeq}. This is directly analogous to the relation between $H_\eta(\QQ)$ and solutions to rational distance problems (\cref{prop:param}).

However, it is worth highlighting a few key differences between the congruent number problem and the family of rational distance problems we consider.
\begin{itemize}
	\item \textbf{Size of parameter space.} The isomorphism class of $E^{(n)}$ is determined by the class of $n$ in $\QQ^\times/(\QQ^\times)^2$, while the isomorphism class of $H_\eta$ is determined by the class of a corresponding matrix in a double quotient of $\GL_2(\QQ)$.
	\item \textbf{Existence of rational points.} Every $n$ determines an elliptic curve $E^{(n)}$, which has a rational point. By contrast, the genus one curves $H_\eta$ typically have no rational points (\cref{thm:upperlower}).
	\item \textbf{Closure under addition of degenerate points.} In both problems, the corresponding genus one curve has a set of ``degenerate'' rational points, which do not yield valid solutions to the original problem. For the congruent number problem, the set of degenerate points equals the torsion subgroup of $E^{(n)}(\QQ)$. For rational configuration problems, however, even if $\mathcal{H}_\eta$ is isomorphic to an elliptic curve (\cref{prop:isomorphicelliptic}), the degenerate points in ${H}_\eta(\QQ)$ may not form a subgroup. This is to our advantage: we can often add together degenerate points to produce non-degenerate points, something that is not possible in the congruent number problem. This is the key idea behind \cref{thm:Qinfpoints}.
	\item \textbf{Geometric variation in the family.} The curves $E^{(n)}$ are quadratic twists of the curve $y^2=x^3-x$, and are therefore all isomorphic over $\overline{\QQ}$. This fact is used in a key way in the proof of Tunnell's theorem, as he applies a result due to Waldspurger~\cite{waldspurger} relating the central value of the $L$-function of an elliptic curve with that of each of its quadratic twists. By contrast, the curves $H_\eta$ do not have constant $j$-invariant. This means that Tunnell's approach to computing the analytic rank does not apply to this family.
\end{itemize}

\section{The structure of the family}\label{sec:setup}

\subsection{Assumptions and notation}\label{sec:assumptions}

Let $K$ be a field of characteristic not equal to $2$, in which $-1$ is not a square; later we will restrict to $K=\QQ$, but many of our results hold in more generality. Throughout this paper, all schemes will be defined over $K$ unless otherwise indicated, and if $X$ and $Y$ are schemes then $X\times Y:=X \times_K Y$. 

Throughout, $\mathbb{P}^1$ will denote the projective line over $K$, while $\mathbb{P}^2$ will denote a \emph{weighted} projective space over $K$, where the variables $x,y,z$ have weights $1,2,1$, respectively. We use the notation $(x:z)$ and $(x:y:z)$ to denote elements of $\mathbb{P}^1(K)$ and $\mathbb{P}^2(K)$, respectively. That is, for $(x,z)\in K^2\setminus\{(0,0)\}$  we have
\[(x:z)=\{(\lambda x:\lambda z)\mid \lambda\in K^\times\},\]
and for $(x,y,z)\in K^3\setminus\{(0,0,0)\}$ we have
\[(x:y:z)=\{(\lambda x,\lambda^2 y:\lambda z)\mid \lambda\in K^\times\}.\]

Let $\GL_2=\Spec K[a,b,c,d,(ad-bc)^{-1}]$ denote the algebraic group of $2\times 2$ invertible matrices, with identity element $I$. Given a matrix $\eta\in \GL_2(K)$, its transpose will be denoted $\eta^t$. Let $\GO_2$ denote the orthogonal group of $2\times 2$ matrices, that is, the algebraic subgroup of $\base$ defined by the condition that $M\in\base(\overline{K})$ is in $\GO_2(\overline{K})$ if and only if $MM^t=M^tM=I$. 

\subsection{Definition of $\mathcal{H}$ and basic properties}\label{sec:Hdef}
Using the coordinates $\left((x:y:z),\begin{psmallmatrix}
	a&b\\c&d
\end{psmallmatrix}\right)$ on $\PP^2\times \GL_2$, define the variety $\mathcal{H}$ by the equation
\begin{align}\label{eq:Hdef}
	\mathcal{H}:y^2&=(a(z^2-x^2)+b(2xz))^2+(c(z^2-x^2)+d(2xz))^2.
\end{align}
If we let $N:K^2\to K$ be defined by $N(u,v)=u^2+v^2$, then this can equivalently be written 
\begin{align}
	\mathcal{H}:y^2&=N\left(\begin{pmatrix}
		a&b\\c&d
	\end{pmatrix}\begin{pmatrix}
		z^2-x^2\\2xz
	\end{pmatrix}\right).
\end{align}
This variety comes equipped with a morphism $\pi:\mathcal{H}\to\GL_2$, which equips $\mathcal{H}$ with the structure of a flat family of curves. Given $\eta\in \GL_2(K)$, ${H}_\eta$ is the fiber of $\pi$ over $\eta$. 

The generic fiber of $\pi$ is a genus one hyperelliptic curve over the function field $K(a,b,c,d)$, with discriminant
\begin{align}\label{eq:Hdisc}
	\Delta(\mathcal{H})=2^{16} (ad-bc)^4 ((a+d)^2+(b-c)^2)((a-d)^2+(b+c)^2).
\end{align}
The Jacobian variety of this curve is an elliptic curve over $K(a,b,c,d)$, which by classical invariant theory (see for example~\cite{weil,jacgenusone}) has a model
\begin{align}\label{eq:jacobian}
	E:y^2=x^3+(a^2+b^2+c^2+d^2)x^2+(ad-bc)^2x.
\end{align}

We have two commuting involutions on $\mathcal{H}$ as a scheme over $\GL_2$, given by 
\begin{align}\label{eq:sigma}
	\sigma_1:(x:y:z)\mapsto (x:-y:z)\qquad\text{and}\qquad \sigma_2:(x:y:z)\mapsto (-z:y:x),
\end{align}
generating a Klein four-group 
\begin{align}\label{eq:Gdef}
	\Gamma:=\langle \sigma_1,\sigma_2\rangle
\end{align}
acting on $\mathcal{H}$. (Note that $\sigma_2$ is an involution because $y$ has weight $2$, and so $(-x:y:-z)=(x:y:z)$.)

\subsection{Double cosets and reduction}\label{sec:doubcos}

We show that the isomorphism class of $\eta\in\GL_2(K)$ is invariant on double cosets in
\[\GO_2(K)\backslash \GL_2(K)/(K^\times\cdot \GO_2(K)),\]
and use this to show that $\mathcal{H}_\eta$ has a $K$-point if and only if $\eta$ is in the same double coset as $\begin{psmallmatrix}
	1&r\\0&s
\end{psmallmatrix}$ for some $r,s\in K$.

\begin{lem}\label{prop:doublecosiso}
	Let $\eta,\eta'\in\base(K)$. If $\eta'\in K^\times\GO_2(K)\eta\GO_2(K)$, then there is an isomorphism $\tau:\mathcal{H}_\eta\to\mathcal{H}_{\eta'}$ over $K$ that commutes with the action of $\Gamma$.
\end{lem}
\begin{proof}
	Let $\eta'=\lambda r_1\eta r_2^{-1}$, where $\lambda\in K^\times$ and $r_1,r_2\in \GO_2(K)$. Write $r_2=\begin{psmallmatrix}
		u&-v\\\epsilon v&\epsilon u
	\end{psmallmatrix}$, where $u,v\in K$, $\epsilon=\pm 1$, and $u^2+v^2=1$. There exists $s,t\in K$ so that $u=\frac{t^2-s^2}{t^2+s^2}$ and $v=\frac{2st}{t^2+s^2}$. Then for any $(x:y:z)\in\mathcal{H}_{\eta}(\overline{K})$,
	\begin{align*}
		(\lambda(s^2+t^2) y)^2 &=N\left(\lambda(s^2+t^2) \begin{pmatrix}
			a&b\\c&d
		\end{pmatrix}\begin{pmatrix}
			z^2-x^2\\2xz
		\end{pmatrix}\right)\\
		&=N\left(r_1\lambda  \begin{pmatrix}
			a&b\\c&d
		\end{pmatrix}r_2^{-1}\begin{pmatrix}
			t^2-s^2&-2st\\2\epsilon st&\epsilon (t^2-s^2)
		\end{pmatrix}\begin{pmatrix}
			z^2-x^2\\2xz
		\end{pmatrix}\right)\\
		&=N\left(\lambda r_1 \begin{pmatrix}
			a&b\\c&d
		\end{pmatrix}r_2^{-1}\begin{pmatrix}
			(tz-sx)^2-(tx+sz)^2\\2\epsilon (tz-sx)(tx+sz)
		\end{pmatrix}\right).
	\end{align*}
	Thus the map 
	\[\tau:(x:y:z)\mapsto (\epsilon(tx+sz):\lambda(s^2+t^2) y:tz-sx)\]
	defines an isomorphism $\mathcal{H}_{\eta}\to \mathcal{H}_{\eta'}$, and the involutions $y\mapsto -y$ and $(x:z)\mapsto (-z:x)$ are preserved.
\end{proof}

Given $\eta\in\base(K)$, suppose $\eta$ is in the same double coset as an element of the form $\eta'=\begin{psmallmatrix}
	1&r\\0&s
\end{psmallmatrix}\in\GL_2(K)$. We have $(0:1:1)\in\mathcal{H}_{\eta'}(K)$, so by \cref{prop:doublecosiso}, we can conclude that $\mathcal{H}_\eta(K)$ is nonempty. The following lemma gives us the converse result: if $\mathcal{H}_\eta(K)$ is nonempty then $\eta$ is in the same double-coset as a matrix of the form $\eta'=\begin{psmallmatrix}
1&r\\0&s
\end{psmallmatrix}$.

\begin{lem}\label{lem:solubletriangular}
	Let $\eta=\begin{psmallmatrix}
		a&b\\c&d
	\end{psmallmatrix}\in\GL_2(K)$. Suppose there is a point $P=(x_0:y_0:z_0)\in\mathcal{H}_\eta(K)$. Define
	\begin{align}\label{eq:cddef}
		\begin{aligned}
			r&:=\frac{(a b + c d) ((z_0^2-x_0^2)^2 - (2x_0z_0)^2) - (a^2 - b^2 + c^2 - d^2) (z_0^2-x_0^2)(2x_0z_0)}{y_0^2},\\
			s&:=\frac{(ad-bc)(z_0^2+x_0^2)^2}{y_0^2}.
		\end{aligned}
	\end{align}
	Then $\eta\in K^\times \GO_2(K)\begin{psmallmatrix}
		1&r\\0&s
	\end{psmallmatrix}\GO_2(K)$.
\end{lem}
\begin{proof}
	Suppose $x_0^2+z_0^2=0$. If $z_0\neq 0$, then $\left(\frac{x_0}{z_0}\right)^2=-1$, contradicting the assumption that $-1$ is not a square in $K$. Hence $z_0=0$, and likewise $x_0=0$. But this implies $y_0=0$, which contradicts the fact that $(x_0:y_0:z_0)\in \widetilde{\PP^2}(K)$.
	
	If $y_0=0$, then a similar argument shows that we must have
	\[a(z_0^2-x_0^2)+b(2x_0z_0)=c(z_0^2-x_0^2)+d(2x_0z_0)=0.\]
	But this implies that the nonzero vector $(z_0^2-x_0^2,2x_0z_0)$ is in the kernel of $\begin{psmallmatrix}
		a&b\\c&d
	\end{psmallmatrix}$, contradicting the assumption that $\eta\in\GL_2(K)$. Hence $y_0\neq 0$.
	
	Since $x_0^2+z_0^2\neq 0$ and $y_0\neq 0$, the matrices
	\begin{align*}
		r_1&=\frac{1}{y_0}\begin{pmatrix}
			a(z_0^2-x_0^2)+b(2x_0z_0)&c(z_0^2-x_0^2)+d(2x_0z_0)\\-c(z_0^2-x_0^2)-d(2x_0z_0)&a(z_0^2-x_0^2)+b(2x_0z_0)
		\end{pmatrix}\\
		r_2&=\frac{1}{z_0^2+x_0^2}\begin{pmatrix}
			z_0^2-x_0^2&-2x_0z_0\\2x_0z_0&z_0^2-x_0^2
		\end{pmatrix}
	\end{align*}
	are both well-defined elements of $\SO_2(K)$. We can check by direct computation that $\frac{z_0^2+x_0^2}{y_0} r_1\eta r_2=\begin{psmallmatrix}
		1&r\\0&s
	\end{psmallmatrix}$.
\end{proof}

\subsection{Isomorphism classes of fibers}\label{sec:singular}

The curve $H_\eta$ is singular when the discriminant (\cref{eq:Hdisc}) vanishes. Since $ad-bc\neq 0$ for all $\eta\in \GL_2(K)$ and $K$ does not contain a square root of $-1$, this can only occur if $a=-d$ and $b=c$, or if $a=d$ and $b=-c$. One of these two conditions holds if and only if $a^2+b^2=c^2+d^2$ and $ac+bd=0$; thus the singular fibers $H_\eta$ are exactly those with 
\[\eta\eta^t=(a^2+b^2)I.\]
In this case $H_\eta$ reduces to the form 
\[y^2=(a^2+b^2)(z^2+x^2)^2.\]
If $a^2+b^2=\lambda^2$ then $H_\eta$ splits into two conics, $y=\pm \lambda(z^2+x^2)$. If $a^2+b^2$ is not a square in $K$, then there are no solutions in $\mathbb{P}^2(K)$.

If $H_\eta$ has a rational point and the discriminant (\cref{eq:Hdisc}) does not vanish at $\eta$, then $H_\eta$ is isomorphic to its Jacobian. Using \cref{lem:solubletriangular} and \cref{eq:jacobian}, we can conclude that $H_\eta$ is isomorphic to
\[E_{r,s}:y^2=x^3+(1+r^2+s^2)x^2+s^2x\]
for some $r,s\in K$; the non-vanishing of the discriminant says that $s\neq 0$ and $(r,s)\neq (0,\pm 1)$. This completes the proof of \cref{prop:isomorphicelliptic}.

\subsection{Nonsingular fibers with a rational point}\label{sec:nonsingular}

We now restrict our attention to $K=\QQ$ in order to prove \cref{thm:Qinfpoints}, which we recall for convenience.

\Qinfpointsrestate*

\begin{proof}
	Assume $R:=(-1,r)$ is torsion in $E_{r,s}(\QQ)$.
	Since $x^2+(1+r^2+s^2)x+s^2$ is positive on an open interval around $x=0$, there exists $-1<x<0$ for which $E_{r,s}(\RR)$ does not contain any point of the form $(x,y)$. Thus $E_{r,s}(\RR)$ has two components, with $R:=(-1,r)$ on the non-identity component and $T:=(0,0)$ on the identity component. This shows $R$ is not a multiple of $2$ in $E_{r,s}(\RR)$, and hence $R$ cannot have odd order. By Mazur's classification of torsion subgroups, we can conclude that if $R$ is torsion then its order must be an even number at most $12$. If $R$ has order $10$ then the only possibility for the torsion subgroup of $E_{r,s}(\QQ)$ is $\Z/{10}\Z$, so that $T$ is the unique element of order $2$. This implies $T=5R$, which again leads to a contradiction when we consider the component group of $E_{r,s}(\RR)$.
	
	We can conclude that if $R$ is torsion, it must have order $\ell\in \{2,4,6,8,12\}$. For each such $\ell$, let $\psi_\ell(r,s,x)\in\ZZ[r,s,x]$ denote the $\ell$-th division polynomial on $E_{r,s}$; this is a polynomial with the property that $\psi_\ell(r,s,x)=0$ for $x\in\overline{\QQ}$ if and only if $(x,y)\in E_{r,s}(\overline{\QQ})[\ell]$ for some $y\in\overline{\QQ}$ (see for instance \cite[Exercise 3.7]{silverman}). We compute the division polynomial $\psi_\ell(r,s,x)$, and determine all possible $(r,s)\in\QQ^2$ such that $\psi_\ell(r,s,-1)=0$. 
	\begin{itemize}
		\item We have $\psi_2(r,s,-1)=-r^2$, so $R$ has order $2$ if and only if $r=0$.
		\item We have
		\[\frac{\psi_4(r,s,-1)}{\psi_2(r,s,-1)}=-2 (s-1)(s+1) (2r^2s^2+2r^2 + (s^2-1)^2).\]
		The last factor is a sum of non-negative terms, including at least one positive term because $(r,s)\neq (0,\pm 1)$. Hence $R$ has order $4$ if and only if $s=\pm 1$.
		\item The quotient of $\psi_6(r,s,-1)$ by $\psi_2(r,s,-1)\psi_3(r,s,-1)$ factors into two irreducible  polynomials in $\QQ[r,s]$. The first factor is $4r^2s^2+(s^2-1)^2$, which is positive for all $(r,s)\neq(0,\pm 1)$. The second factor is
		\[16s^2r^4-4(s^2-1)^2(s^2+1)r^2-3(s^2-1)^4.\]
		Considering this as a quadratic polynomial in $r^2$, the discriminant is equal to 
		\[16(s^2-1)^4(s^4+14s^2+1).\]
		In order for $r^2$ to be rational (let alone $r$), this discriminant must equal a rational square. Thus we consider rational points on the curve $C$ defined by $y^2=s^4+14s^2+1$. There are eight rational points $(s,y)\in C(\QQ)$: two at infinity, as well as $(-1,\pm 4)$, $(0,\pm 1)$, and $(1,\pm 4)$. Using the Weierstrass form $y^2=x^3-7x^2+12x$ for $C$ we can confirm that $C$ has no other rational points, so the only possibilities for $s$ are $-1,0,1$. If $s=0$ then we have $r^2=-\frac 34$, yielding no rational solutions. If $s=\pm 1$ then we have $r=0$, contradicting $(r,s)\neq(0,\pm 1)$.
		\item The quotient of $\psi_8(r,s,-1)$ by $\psi_4(r,s,-1)$ factors into three irreducible polynomials in $\QQ[r,s]$. 
		The first two factors are $4r^2s-(s^2-1)^2$ and $4r^2s+(s^2-1)^2$; these each have infinitely many rational solutions. The third factor is
		positive for all $(r,s)\neq(0,\pm 1)$.
		\item If we eliminate common factors with $\psi_6(r,s,-1)$ and $\psi_4(r,s,-1)$ from $\psi_{12}(r,s,x)$, we are left with three irreducible polynomials in $\QQ[r,s]$. The first factor is 
		\[16s(s^2-s+1)r^4+8s(s^2-1)^2r^2+(s^2-1)^4.\]
		Considered as a quadratic in $r^2$, the discriminant is $-64s(s-1)^6(s+1)^4$, which is a square if and only if $s=-k^2$ for some $k\in\QQ$. Plugging this in and solving for $r^2$, we find that either
		\[r^2=\frac{(k^4-1)^2}{4 k (k^2+k+1)}\qquad\text{or}\qquad r^2=-\frac{(k^4-1)^2}{4 k (k^2-k+1)}.\]
		For the first option, we obtain $r\in\QQ$ if and only if $k^3+k^2+k$ is a nonzero square. The only rational points on the elliptic curve $y^2=k^3+k^2+k$ are the point at infinity and $(k,y)=(0,0)$, so there is no $k\in\QQ$ for which $r$ is rational. For the second option, we obtain $r\in\QQ$ if and only if $(-k)^3+(-k)^2+(-k)$ is a nonzero square, and by the same reasoning there is no such $k$. Hence this factor is nonzero for all $(r,s)\in\QQ^2$.
		
		The second factor is obtained from the first by $s\mapsto -s$, so it has no rational solutions either.
		The third factor is
		positive for all $(r,s)\neq(0,\pm 1)$.
	\end{itemize}
	To summarize, we obtain the following possibilities:
	\begin{itemize}
		\item $(-1,r)$ has order $2$ if and only if $r=0$;
		\item $(-1,r)$ has order $4$ if and only if $s=\pm 1$;
		\item $(-1,r)$ has order $8$ if and only if $4r^2s=\pm (1-s^2)^2$.
		\item There are no values of $(r,s)\in\QQ^2$ for which $(-1,r)$ has any other finite order.\qedhere
	\end{itemize}
\end{proof}

\section{Upper bound on locally soluble curves}\label{sec:upperbound}

In this section we prove \cref{thm:upperlower}. The main idea is to prove a local obstruction to the existence of $\QQ_p$-points (\cref{lem:localobstruct}). This obstruction is ``large,'' in the sense that the proportion of curves satisfying the obstruction is approximately a constant multiple of $\frac 1p$. By contrast, each local obstruction in \cref{bigfamilycompare} only affects $O(\frac1{p^2})$ of all curves. At a high level, the difference in behavior between the two families stems from the fact that $\sum\frac1p$ diverges but $\sum \frac1{p^2}$ converges.

\subsection{Local Obstructions}

Given a ring $A$ we use $M_2(A)$ to denote the ring of $2\times 2$ matrices over $A$. 
For primes $p$ we define
\[R_p:=\left\{\eta\in M_2(\Z_p)\cap \GL_2(\Q_p):H_\eta(\Q_p)=\emptyset\right\};\]
Thus $\mathcal{L}(X)$ counts the set of $\eta\in M_2(\Z)\cap\GL_2(\Q)$ with entries of absolute value at most $X$ such that $\eta\notin R_p$ for all primes $p$ (note that there is never a real obstruction: $H_\eta(\RR)\neq\emptyset$ for all $\eta\in\GL_2(\Q)$.)

The main contribution to $R_p$ will come from the following constraint.
\begin{lem}\label{lem:localobstruct}
	Let $p$ be an odd prime and $\eta=\begin{psmallmatrix}
		a&b\\c&d
	\end{psmallmatrix}\in M_2(\Z_p)\cap \GL_2(\Q_p)$. Suppose $p\mid ad-bc$, and $a$, $a^2+b^2$, and $a^2+c^2$ are all nonzero mod $p$. Then $H_\eta(\Q_p)$ is nonempty if and only if at least one of $a^2+b^2$ or $a^2+c^2$ is a square modulo $p$.
\end{lem}
\begin{proof}
	Assume $H_\eta(\Q_p)$ has a point $(x:y:z)$; without loss of generality we can assume that $x,y,z$ are in $\Z_p$ and at least one of $x,z$ is in $\Z_p^\times$. Reducing modulo $p$ we have
	\begin{align}\label{eq:modp}
		{a}^2{y}^2=({a}^2 + {c}^2) ({a}({z}^2-{x}^2) + {b}(2{x}{z}))^2.
	\end{align}
	If ${a}({z}^2-{x}^2) + {b}(2{x}{z})\neq 0$, then 
	\[{a}^2 + {c}^2=\left(\frac{a}{y}{{a}({z}^2-{x}^2) + {b}(2{x}{z})}\right)^2\]
	is a square. On the other hand, suppose ${a}({z}^2-{x}^2) + {b}(2{x}{z})=0$. Since $a\neq 0$ and at least one of $x,z$ is nonzero we must have $xz\neq 0$, so
	\begin{align*}
		{a}^2+{b}^2&={a}^2+\left(\frac{-{a}({z}^2-{x}^2)}{2{x}{z}}\right)^2=\left(\frac{{a}({z}^2+{x}^2)}{2{x}{z}}\right)^2
	\end{align*}
	is a square.
	
	Conversely, if $a^2+c^2$ is a nonzero square mod $p$ then \cref{eq:modp} clearly has solutions with ${y}\neq 0$; these are smooth points on $H_\eta(\F_p)$ so they lift to points on $H_\eta(\Q_p)$. If $a^2+b^2$ is a nonzero square mod $p$, then it is a square in $\Q_p$, so there exist $x,z\in\Q_p^\times$ with $\frac{b}{a}=\frac{x^2-z^2}{2xz}$. This implies $a(z^2-x^2)+b(2xz)=0$, so $(x:c(z^2-x^2)+d(2xz):z)$ is a point in $H_\eta(\Q_p)$.
\end{proof}

In light of the above, set 
\begin{align*}
	\Omega_p:=\left\{\mu=\begin{pmatrix}{a}&{b}\\{c}&{d}\end{pmatrix}\in M_2(\F_p):{a}{d}-{b}{c}=0,\; \bigg( \frac{{a}^2+{b}^2}p\bigg)=-1,\; \bigg( \frac{{a}^2+{c}^2}p\bigg)=-1\right\},
\end{align*}
where $\left(\frac{\cdot}{p}\right)$ denotes the Legendre symbol.
If by abuse of notation we associate $\Omega_p$ with its preimage in $M_2(\Z_p)$ under reduction mod $p$, we have
\begin{align}\label{eq:containment}
	(\Omega_p\cap \GL_2(\Q_p))\subseteq R_p.
\end{align}
This follows from \cref{lem:localobstruct}: note that the Legendre symbol conditions in the definition of $\Omega_p$ force ${a}$ to be nonzero mod $p$.

\begin{lem}\label{countmodp}
	We have $|\Omega_p|=\frac{1}{4}p^3+O(p^2)$.
\end{lem}
\begin{proof}
	Let $\begin{psmallmatrix}
		{a}&{b}\\{c}&{d}
	\end{psmallmatrix}\in \Omega_p$. Since ${a}^2+{b}^2$ is not a square, we have ${a}\neq 0$. Note that ${a}^2+{r}^2$ is a square in $\F_p$ if and only if ${r}$ is in the image of the map $\phi:\F_p^\times\to\F_p$ given by $\phi({t})= {a}\frac{1-{t}^2}{2{t}}$. We have $\phi({t})=\phi({s})$ if and only if ${s}=-\frac{1}{{t}}$, and so there are $\frac{1}{2}(p\pm 1)$ values in the range of $\phi$ (with the $\pm$ sign depending on whether or not $-1$ is a square mod $p$). Hence ${b}$ and ${c}$ must each be one of the $\frac{1}{2}(p\mp 1)$ values not in the range of $\phi$. Since there are $p-1$ choices for ${a}$, there are $\frac{1}{2}(p\mp 1)$ choices for each of ${b}$ and ${c}$, and the value of ${d}=\frac{{b}{c}}{{a}}$ is fixed, we obtain the desired count.
\end{proof}

The following result is not used in the sequel, but justifies the claim that $\Omega_p$ is the main contribution to $R_p$.

\begin{lem}
	Let $\overline{R_p}$ denote the image of $R_p$ under reduction mod $p$. We have
	\[|\overline{R_p}\setminus\Omega_p|=O(p^2).\]
\end{lem}
\begin{proof}
	We produce a collection of pairs of linear equations over $\F_p$ with the property that every element of $\overline{R_p}\setminus\Omega_p$ satisfies one of these pairs of equations. Let $\bar{\eta}=\begin{psmallmatrix}
		a&b\\c&d
	\end{psmallmatrix}\in M_2(\F_p)$.
	If 
	\[(ad-bc)((a-d)^2+(b+c)^2)((a+d)^2+(b-c)^2)\neq 0,\]
	then for any lift $\eta\in M_2(\Z_p)\cap \GL_2(\Q_p)$, the discriminant of $H_\eta$ (\cref{eq:Hdisc}) is not divisible by $p$, and so $H_\eta(\F_p)$ is nonempty by the Hasse-Weil bound. As these are all smooth points, they lift to points in $H_\eta(\Q_p)$. Thus $\bar{\eta}$ is not in $\overline{R_p}$.
	
	We can therefore assume that exactly one of the following constraints holds:
	\begin{enumerate}[label=(\alph*)]
		\item ${a}{d}-{b}{c}=0$ and ${a}=0$,
		\item ${a}{d}-{b}{c}=0$ and ${a}\neq 0$,
		\item ${a}{d}-{b}{c}\neq 0$, $({a}\mp {d})^2+({b}\pm {c})^2=0$, and ${a}\mp {d}=0$,
		\item ${a}{d}-{b}{c}\neq 0$, $({a}\mp {d})^2+({b}\pm {c})^2=0$ and ${a}\mp {d}\neq 0$.
	\end{enumerate}
	In case (a) or (c), $\bar{\eta}$ must satisfy one of the following pairs of linear equations over $\F_p$:
	\[{a}={b}=0,\qquad {a}={c}=0,\qquad\text{or}\qquad {a}\mp{d}={b}\pm{c}=0.\]
	Now consider case (b). If ${a}^2+{b}^2=0$ or ${a}^2+{c}^2=0$, then using ${a}\neq 0$ and ${a}{d}={b}{c}$ we can conclude that
	\[{a}+{i}{b}={c}+{i}{d}=0\qquad\text{or}\qquad {a}+{i}{c}={b}+{i}{d}=0\]
	for some ${i}\in\F_p$ satisfying ${i}^2=-1$. If on the other hand ${a}^2+{b}^2$ and ${a}^2+{c}^2$ are both nonzero, then $\bar{\eta}$ is not in $\overline{R_p}$ by \cref{lem:localobstruct}.
	
	Finally we consider case (d). We assume $(a-d)^2+(b+c)^2=0$ and $a-d\neq 0$ as the other case is similar. Then $b+c= i(a-d)$ for some $i\in\F_p$ with $i^2=-1$. The reduction of $H_\eta$ modulo $p$ is given by \[y^2=(z + i x)^2 h(x,z)\]
	where
	\[h(x,z):=(a^2 + c^2) (z^2 - x^2) + (2 a (b + c) + 
	i (c^2 - a^2 + 2 b c)) (2 x z).\]
	If the discriminant of $h(x,z)$ is nonzero, then the equation $r^2=h(x,z)$ defines a smooth projective conic over $\F_p$, and there must exist at least one point $(x,z,r)$ on this conic with $z+ix\neq 0$. Then $(x:(z+ix)r:z)$ defines a smooth point in $H_\eta(\F_p)$, which lifts to a point in $H_\eta(\Q_p)$. Hence $\bar{\eta}$ is not in $\overline{R_p}$. Computing the discriminant of $h(x,z)$, we find that elements in $\overline{R_p}\setminus\Omega_p$ in this case must satisfy
	\[b+c-i(a-d)=(b- i a) (a + i c) (b + c)=0.\]
\end{proof}

\subsection{Proof of \cref{thm:upperlower}} 

The author would like to thank Sun-Kai Leung for suggesting the following proof.

Recall that $\mathcal{L}(X)$ is the set of all $\eta\in M_2(\Z)$ with nonzero determinant, entries having absolute value at most $X$, and with $\eta\notin R_p$ for all $p$. Set
\[\mathcal{M}(X):=\left\{\eta\in M_2(\Z)\cap\GL_2(\Q):\eta\notin \Omega_p\text{ for all }p\right\}.\]
By \cref{eq:containment} we can see that $\mathcal{L}(X)\subseteq \mathcal{M}(X)$, and so it suffices to find an upper bound for $\mathcal{M}(X)$. 

Set
\begin{align*}
	\mathscr{S}(Y)&:= \sum_{m\leq Y}\mu^2(m)\prod_{\text{prime }p\mid m}\frac{|\Omega_p|}{p^4-|\Omega_p|},
\end{align*}
where $\mu(m)$ is the M\"obius function (so $\mu^2(m)=1$ if $m$ is squarefree and $\mu^2(m)=0$ otherwise). Applying the $n$-dimensional large sieve \cite[Lemma B]{gallagher}, we obtain
\begin{align*}
	|\mathcal{L}(X)|\leq|\mathcal{M}(X)|\ll \frac{X^4}{\mathscr{S}(\sqrt{X})},
\end{align*}
where $f(X)\ll g(X)$ means that for some positive constant $C$ we have $f(X)<Cg(X)$ for sufficiently large $X$. \cref{thm:upperlower} follows immediately from this bound after applying the following lemma.

\begin{lem}
	We have
	\[\mathscr{S}(Y)\gg (\log Y)^{1/4}.\]
\end{lem}
\begin{proof}
	We have $|\Omega_p|\geq 0$, and by \cref{countmodp} we have $|\Omega_p|\geq \frac14(p^3-Cp^2)$ for some positive constant $C$, so
	\begin{align*}
		\mathscr{S}(Y)&\geq \sum_{m\leq Y}\mu^2(m)\prod_{\text{prime }p\mid m}\frac{p^3-Cp^2}{4p^4}\\
		&=\sum_{m\leq Y}\frac{\mu^2(m)}{m}\prod_{p\mid m}\frac14\left(1-\frac Cp\right).
	\end{align*}
	Set 
	\[f(m):= \mu^2(m)\prod_{\text{prime }p\mid m}\frac14\left(1-\frac Cp\right).\]
	Note that $f$ is a multiplicative function, it is positive for sufficiently large $p$, and
	\begin{align*}
		\sum_{\text{prime }p\leq Y}\frac{f(p)\log p}{p}&=\frac14\sum_{p\leq Y}\frac{\log p}{p}-C\sum_{p\leq Y}\frac{\log p}{p^2}\\
		&=\frac14\log(Y)+O(1)
	\end{align*}
	(see for instance \cite[p.~57]{davenport}). Hence, by Wirsing's Theorem (\cite[Theorem 14.3]{koukoulopoulos} with $\kappa=\frac14$, $c=0$, $k=1$) we have
	\begin{align*}
		\mathscr{S}(Y)&\geq \sum_{m\leq Y}\frac{f(m)}{m}\gg (\log Y)^{1/4}.
	\end{align*}
\end{proof}

\section{Applications to Rational Distance Problems}\label{sec:applications}

\subsection{From $H_\eta(K)$ to rational configurations}

We return temporarily to the more general setting of a field $K$ in which $-1$ is not a square. Let $\eta=\begin{psmallmatrix}
	a&b\\c&d
\end{psmallmatrix}\in\GL_2(K)$. In \cref{eq:Hdef}, we defined the variety
\begin{align*}
	\mathcal{H}:y^2&=N\left(\begin{pmatrix}
		a&b\\c&d
	\end{pmatrix}\begin{pmatrix}
		z^2-x^2\\2xz
	\end{pmatrix}\right),
\end{align*}
where $N(u,v):=u^2+v^2$. We also define the subvariety $\mathcal{F}$ of $\mathbb{P}^1\times\mathbb{P}^1\times \GL_2$ (with coordinates $((u_1:v_1),(u_2:v_2),\begin{psmallmatrix}
	a&b\\c&d
\end{psmallmatrix})$ by
\[\mathcal{F}:au_1u_2+bu_1v_2+cv_1u_2+dv_1v_2=0.\]
As with $\mathcal{H}$, there is a projection map to $\GL_2$ giving $\mathcal{F}$ the structure of a flat family of curves, with $F_\eta$ denoting the fiber over $\eta\in\GL_2(K)$. We define
\begin{align*}
	\mathcal{S}_K&:=\{(u:v)\in \mathbb{P}^1(K)\mid u^2+v^2\text{ is a square in }K\}\\
	&=\{(z^2-x^2:2xz)\mid (x:z)\in \mathbb{P}^1(K)\}.
\end{align*}
\begin{prop}\label{prop:bijection}
	There is a morphism $\Phi:\mathcal{H}\to \mathcal{F}$ over $\GL_2$ defined by sending $(x:y:z)$ to
	\begin{align}
		\left((-c(z^2-x^2)-d(2xz):a(z^2-x^2)+b(2xz)),\,(z^2-x^2:2xz)\right).
	\end{align}
	Further, $\Phi$ induces a bijection between the set of $\Gamma$-orbits in ${H}_\eta(K)$ and the set $F_\eta(K)\cap(\mathcal{S}_K\times\mathcal{S}_K)$.
\end{prop}
\begin{proof}
	Notice that the equation defining $\mathcal{F}$ can be written
	\begin{equation}
		\begin{aligned}
			0&=\begin{pmatrix}
				u_1&v_1
			\end{pmatrix}\begin{pmatrix}
				a&b\\c&d
			\end{pmatrix}\begin{pmatrix}
				u_2\\v_2
			\end{pmatrix}.
		\end{aligned}
	\end{equation} 
	The morphism $\Phi$ is well-defined because
	\begin{align*}
		&\left(\begin{pmatrix}
			-c&-d\\a&b
		\end{pmatrix}\begin{pmatrix}
			z^2-x^2\\2xz
		\end{pmatrix}\right)^t\begin{pmatrix}
			a&b\\c&d
		\end{pmatrix}\begin{pmatrix}
			z^2-x^2\\2xz
		\end{pmatrix}\\
		&\qquad=\begin{pmatrix}
			z^2-x^2&2xz
		\end{pmatrix}\begin{pmatrix}
			0&ad-bc\\-ad+bc&0
		\end{pmatrix}\begin{pmatrix}
			z^2-x^2\\2xz
		\end{pmatrix}\\
		&\qquad=0.
	\end{align*}
	Given $(x:y:z)\in\mathcal{H}_\eta(K)$, we have 
	\[(z^2-x^2)^2+(2xz)^2=(z^2+x^2)^2\]
	and
	\[(-c(z^2-x^2)-d(2xz))^2+(a(z^2-x^2)+b(2xz))^2=y^2,\]
	so that $\Phi(x:y:z)\in\mathcal{S}_K\times\mathcal{S}_K$.
	Conversely, given any $(\alpha_1,\alpha_2)\in F_\eta(K)\cap (\mathcal{S}_K\times \mathcal{S}_K)$, we can write $\alpha_2=(z^2-x^2:2xz)$ and $\alpha_1=(z'^2-x'^2:2x'z')$ for some $x,z,x',z'\in K$. The fact that $(\alpha_1,\alpha_2)\in F_\eta(K)$ is then equivalent to 
	\[\begin{pmatrix}
		z'^2-x'^2&2x'z'
	\end{pmatrix}\begin{pmatrix}
		a&b\\c&d
	\end{pmatrix}\begin{pmatrix}
		z^2-x^2\\2xz
	\end{pmatrix}=0.\]
	This implies that $\begin{psmallmatrix}
		a&b\\c&d
	\end{psmallmatrix}\begin{psmallmatrix}
		z^2-x^2\\2xz
	\end{psmallmatrix}$ must equal $\lambda\begin{psmallmatrix}
		-2x'z'\\ z'^2-x'^2
	\end{psmallmatrix}$ for some $\lambda\in K^\times$. In particular, 
	\[N\left(\begin{pmatrix}
		a&b\\c&d
	\end{pmatrix}\begin{pmatrix}
		z^2-x^2\\2xz
	\end{pmatrix}\right)=\lambda^2(z'^2+x'^2)^2,\]
	so that $(x:y:z)\in\mathcal{H}_\eta(K)$ for $y=\lambda(z'^2+x'^2)$. 
	One can then confirm that $\Phi$ maps $(x:y:z)$ to $(\alpha_1,\alpha_2)$, using the computation
	\begin{align*}
		\begin{pmatrix}
			-c&-d\\a&b
		\end{pmatrix}\begin{pmatrix}
			z^2-x^2\\2xz
		\end{pmatrix}&=\begin{pmatrix}
			0&-1\\1&0
		\end{pmatrix}\begin{pmatrix}
			a&b\\c&d
		\end{pmatrix}\begin{pmatrix}
			z^2-x^2\\2xz
		\end{pmatrix}\\
		&=\lambda\begin{pmatrix}
			0&-1\\1&0
		\end{pmatrix}\begin{pmatrix}
			-2x'z'\\z'^2-x'^2
		\end{pmatrix}\\
		&=-\lambda\begin{pmatrix}
			z'^2-x'^2\\2x'z'
		\end{pmatrix}.
	\end{align*}
	Hence $\Phi$ maps ${H}_\eta(K)$ surjectively onto $F_\eta(K)\cap (\mathcal{S}_K\times\mathcal{S}_K)$.
	
	Finally, observe that for each $\alpha_2\in\mathcal{S}_K$, there are two choices for $(x:z)\in \PP^1(K)$ with $(z^2-x^2:2xz)=\alpha_2$, interchanged by the involution $(x:z)\mapsto (-z:x)$. Once $x$ and $z$ are fixed, there are two choices for $y$, interchanged by $y\mapsto -y$. Hence $\Gamma$ acts transitively on the fibers of $\Phi$.
\end{proof}

\begin{rmk}
	For many rational distance problems, solutions $((u_1:v_1),(u_2:v_2))\in F_\eta(K)\cap (\mathcal{S}_K\times\mathcal{S}_K)$ with $u_1v_1u_2v_2=0$ will be considered \emph{degenerate} (as they correspond to rational right triangles with no width). The degenerate locus $u_1v_1u_2v_2=0$ pulls back to the subvariety $\mathcal{D}\subseteq \mathcal{H}$ defined by
	\begin{align}\label{eq:degen}
		\mathcal{D}:xyz(z^4-x^4)(a(z^2-x^2)+b(2xz))(c(z^2-x^2)+d(2xz))=0.
	\end{align}
\end{rmk}

\subsection{Density of rational configuration solutions}

For any embedding $K\hookrightarrow\RR$, if $\mathcal{H}_\eta(K)$ is infinite, we can show that $F_\eta(K)\cap (\mathcal{S}_K\times\mathcal{S}_K)$ is dense in $F_\eta(\RR)$. This is a special case of the following result. Let $E$ denote the Jacobian of $H_\eta$.

\begin{lem}\label{lem:dense}
	Let $\eta\in\GL_2(\RR)$ and suppose $\Delta(\mathcal{H}_\eta)\neq 0$. Let $A\subseteq {H}_\eta(\RR)$ be the image of an infinite subgroup of $E(\RR)$ under some isomorphism $E(\RR)\cong{H}_\eta(\RR)$. Then the image of $A$ under $\Phi:{H}_\eta\to F_\eta$ (\cref{prop:bijection})
	is dense in $F_\eta(\RR)$.
\end{lem}
\begin{proof}
	The (topological) curve ${H}_\eta(\RR)$ has two connected components, given by points $(x:y:z)$ with $y>0$ and those with $y<0$ respectively: there is no equivalence between any points with $y>0$ and points with $y<0$ because of the weighting on ${\PP^2}$ (\cref{sec:Hdef}), and there are no points with $y=0$ because $\Delta(\mathcal{H}_\eta)\neq 0$ and $-1$ is not a square in $K$. Thus $E(\RR)$ has structure of a Lie group $S^1(\RR)\times \ZZ/2\ZZ$. Any infinite subgroup of $E(\RR)$ has dense intersection with the identity component, so $A$ has dense intersection with one of the components of ${H}_\eta(\RR)$.
	
	We can express the map ${\Phi}:{H}_\eta\to{F}_\eta$ as a composition
	\[\begin{array}{ccccc}
		{H}_\eta&\to&\PP^1 & \to& {F}_\eta \\
		(x:y:z)&\mapsto& (x:z)&\mapsto& \left(\begin{gathered}
			(-c(z^2-x^2)-d(2xz):a(z^2-x^2)+b(2xz)),\\ (z^2-x^2:2xz)
		\end{gathered}\right).
	\end{array}\]
	The first map induces a continuous surjection from each component of $\mathcal{H}_\eta(\RR)$ onto $\PP^1(\RR)$, so the image of $A$ is dense in $\PP^1(\RR)$. The second map induces a continuous surjection $\PP^1(\RR)\to {F}_\eta(\RR)$, so the image of $A$ is dense in ${F}_\eta(\RR)$.
\end{proof}

\subsection{Application to three-distance problem}\label{sec:threedistproof}

We will use \cref{thm:Qinfpoints} to prove the following statement.

\begin{cor}
	\label{cor:squaredist_param}
	There exists an infinite collection of rational functions, $\rho_n:\bbA^1_\QQ\dashrightarrow \bbA^2_\QQ$ for $n\in\ZZ$, with the following properties. For all $t\in\QQ-\{0,\pm 1\}$ and all $n\in\ZZ$, if $\rho_n$ is defined at $t$, then $\rho_n(t)$ has rational distance from each of $(0,0)$, $(0,1)$, and $(1,1)$. Further, for each $t\in\QQ-\{0,\pm 1\}$, there are only finitely many $n\in\ZZ$ for which $\rho_n$ is not defined at $t$, and the set 
	\[\{\rho_n(t):n\in \ZZ,\,\rho_n\text{ defined at }t\}\]
	is a dense subset of the line $y=\frac{2t}{1-t^2}x$ in $\RR^2$. 
\end{cor}
\begin{proof}
	For the sake of clarity, we begin by proving the weaker result mentioned in the introduction: for each $t\in\QQ-\{0,\pm 1\}$, the line $y=\frac{2t}{1-t^2}x$ has a dense set of points that have rational distance from each of $(0,0)$, $(0,1)$, and $(1,1)$. Once this is done, we will explain how the proof can be modified to allow for families of solutions parametrized by $t$.
	
	Let $t\in\QQ-\{0,\pm 1\}$, and set $s(t):=1-\frac{2t}{1-t^2}$. There is no rational solution to $1=\frac{2t}{1-t^2}$, so $\eta(t):=\begin{psmallmatrix}
		1&-1\\0&s(t)
	\end{psmallmatrix}$ is an element of $\GL_2(\QQ)$. We have $|s(t)|\neq 1$ because we excluded the case $t=0$ and there is no rational solution to $2=\frac{2t}{1-t^2}$. Further, there is no rational solution to $\left|\frac{1-u^4}{2u}\right|=1$. Hence, by \cref{thm:Qinfpoints}, $\mathcal{H}_{\eta(t)}(\QQ)$ is infinite. By \cref{lem:dense}, this implies that the set of $((u_1:v_1),(u_2:v_2))\in\mathcal{S}\times \mathcal{S}$ satisfying $u_1u_2+v_1v_2=u_1v_2+\frac{2t}{1-t^2}v_1v_2$ (the defining equation of $F_{\eta(t)}$) is dense in $F_{\eta(t)}(\RR)$.

	Now define the rational function $z:F_{\eta(t)}\dashedrightarrow\bbA^2$ by  
	\begin{align}\label{eq:defz}
		z((u_1:v_1),(u_2:v_2)):=\left(\frac{(1-t^2)v_1}{(1-t^2)u_1+2tv_1},\,\frac{2tv_1}{(1-t^2)u_1+2tv_1}\right).
	\end{align}
	The map $z$ restricts to a homeomorphism
	\[F_{\eta(t)}(\RR)\setminus\left\{\left((-2t:1-t^2),(1-t^2:2t)\right)\right\}\to \left\{(x,y)\in\RR^2:y=\tfrac{2t}{1-t^2}x\right\},\]
	So after removing a single point from $F_{\eta(t)}(K)\cap (\mathcal{S}\times\mathcal{S})$, the remainder maps to a dense subset of the line $y=\frac{2t}{1-t^2}x$. For each $(\alpha_1,\alpha_2)\in F_{\eta(t)}(\QQ)\cap (\mathcal{S}\times\mathcal{S})$ other than $\left((-2t:1-t^2),(1-t^2:2t)\right)$, the point $(x,y):=z(\alpha_1,\alpha_2)$ satisfies
	\begin{align*}
		x^2+y^2&=\left(\frac{(1+t^2)v_1}{u_1(1-t^2)+2tv_1}\right)^2,\\
		x^2+(1-y)^2&=\left(\frac{(1-t^2)}{u_1(1-t^2)+2tv_1}\right)^2(u_1^2+v_1^2),\\
		(1-x)^2+(1-y)^2&=\left((1-t^2)\frac{u_1+\frac{2t}{1-t^2}v_1-v_1}{(1-t^2)u_1+2tv_1}\right)^2+\left(\frac{(1-t^2)u_1}{(1-t^2)u_1+2tv_1}\right)^2\\
		&=\left(\frac{(1-t^2)u_1u_2}{((1-t^2)u_1+2tv_1)v_2}\right)^2+\left(\frac{(1-t^2)u_1}{(1-t^2)u_1+2tv_1}\right)^2\\
		&=\left(\frac{(1-t^2)u_1}{((1-t^2)u_1+2tv_1)v_2}\right)^2(u_2^2+v_2^2).
	\end{align*}
	Since $\alpha_1,\alpha_2\in\mathcal{S}$, these are all squares in $\QQ^\times$, so this gives a solution to the three-distance problem.

	We now return to the problem of producing explicit parametrizations of solutions in terms of $t$. For this, note that $t\mapsto \eta(t)$ defines a morphism $V:=\bbA^1-\{0,\pm 1\}\to\GL_2$. We will define a rational map $\rho_n:V\dashedrightarrow\bbA^2$ by a composition
	\[\rho_n:V\xrightarrow{\tau_n} E'\xrightarrow{\varepsilon} \mathcal{H}'\xrightarrow{\Phi'} \mathbb{P}^1\times \mathbb{P}^1\times V\xrightarrow{z'} \bbA^2,\]
	where each variety besides $\bbA^2$ is a scheme over $V$ and each map besides $z'$ is a morphism over $V$. We consider each of these maps in turn.
	\begin{itemize}
		\item Let $E$ be the subvariety of $\PP^2\times\GL_2$ parametrizing the Jacobian varieties of $\mathcal{H}$ (defined by \cref{eq:jacobian}). Let $E'$ be the fiber product of $V$ with $E$, so that $E'_t=E_{\eta(t)}$ for all $t\in V(\QQ)$. We have a section $V\to E'$ given by $t\mapsto (-1,-1)$. Using the group law on the generic fiber of $E'$, define the rational map $\tau_n:V\dashedrightarrow E'$ by the property that $\tau_n(t)=n(-1,-1)\in E'_t(\QQ)$ for all $t\in V(\QQ)$. The proof of \cref{thm:Qinfpoints} shows that $(-1,-1)$ is non-torsion in $E'_t(\QQ)$ for all $t\in V(\QQ)$, so for each such $t$, the set $\{\tau_n(t):n\in\ZZ\}$ is an infinite subgroup of $E'_t(\QQ)$.
		\item The fiber product of $V$ with $\mathcal{H}$ is a one-parameter family $\mathcal{H}'$ of curves over $V$, with the property that $\mathcal{H}'_t=\mathcal{H}_{\eta(t)}$. We have a section $V\to \mathcal{H}'$ given by $t\mapsto (0:1:1)$, allowing us to define a birational map $\varepsilon:E'\dashedrightarrow \mathcal{H}'$ over $\QQ$ sending the zero section of $E'$ to the given section of $\mathcal{H}'$. This map restricts to an isomorphism on all fibers over points in $V(\QQ)$.
		\item The rational map $\Phi':\mathcal{H}'\dashedrightarrow\mathbb{P}^1\times \mathbb{P}^1\times V$ is defined by
		\[
			((x:y:z),t)\mapsto \left((-d(t)(2xz):(z^2-x^2)-(2xz)),(z^2-x^2:2xz),\,t\right).
		\]
		Note that after restricting to a fiber $\mathcal{H}'_t$, the first two components of $\Phi'$ agree with the map $\Phi:{H}_{\eta(t)}\to F_{\eta(t)}$ (\cref{prop:bijection}). 
		So for any $t\in V(\QQ)$, 
		the set
		\[S_t:=\{(\Phi'\circ\varepsilon\circ \tau_n)(t):n\in\ZZ\}\]
		is contained in $F_{\eta(t)}(\QQ)\cap(\mathcal{S}\times\mathcal{S})\times \{t\}$, and is a dense subset of $F_{\eta(t)}(\RR)\times \{t\}$  by \cref{lem:dense}.
		\item The rational map $z':\PP^1\times\PP^1\times V\dashedrightarrow \bbA^2$ is defined on an appropriate dense open subset by
		\[z'((u_1:v_1),(u_2:v_2),t)=\left(\frac{(1-t^2)v_1}{(1-t^2)u_1+2tv_1},\,\frac{2tv_1}{(1-t^2)u_1+2tv_1}\right).\]
		Note that when restricted to $F_{\eta(t)}\times\{t\}$, the map agrees with $z:F_{\eta(t)}\dashedrightarrow \bbA^2$ defined in \cref{eq:defz}. So the same proof as above shows that for each $t$, $z'$ maps $S_t$ minus a point to a dense subset of the line $y=\frac{2t}{1-t^2}x$ consisting of solutions to the three-distance problem.\qedhere
	\end{itemize}

\end{proof}

\subsection{Special cases}\label{sec:sumprod}

Some rational configuration problems fall under the exceptional cases of \cref{thm:Qinfpoints}. We consider a few of these here. Let
\begin{align*}
	\mathcal{S}':=\left\{\frac{u}{v}\in\QQ\mid (u:v)\in \mathcal{S}\right\}
\end{align*}
be the affine elements of $\mathcal{S}$.

\begin{prop}\label{prop:pythslopesum}
	Let $\alpha_3\in\mathcal{S}'$. There exist infinitely many pairs $\alpha_1,\alpha_2\in\mathcal{S}'$ such that $\alpha_1+\alpha_2=\alpha_3$.
\end{prop}
That is, for any rectangle with rational distances between every two vertices, there are infinitely many ways to split it into two rectangles with rational distances between every two vertices.
\begin{proof}
	Let $\alpha_3=\frac{1-t^2}{2t}$ for some $t\in\QQ-\{0,\pm 1\}$. The equation $x_1+x_2-\alpha_3=0$ defines $F_\eta$ for $\eta=\begin{psmallmatrix}
		0&1\\1&-\alpha_3
	\end{psmallmatrix}$, and $\mathcal{H}_\eta$ (which contains the rational point $(0:1:1)$) is isomorphic to its Jacobian (as in \cref{eq:jacobian}),
	\[E:y^2=x^3+\left(\left(\frac{1-t^2}{2t}\right)^2+2\right)x^2+x.\]
	We consider $E_\eta$ as an elliptic curve over the function field $\QQ(t)$, and note that $E(\QQ(t))$ has a rational point $P=\left(t,\frac12(t+1)^2\right)$. By computing $nP$ for $n=1,\ldots,12$ and checking that the denominators of the coordinates have no rational roots, we can confirm that $P$ is non-torsion for all $t\in\QQ-\{0,\pm 1\}$. Hence $\mathcal{H}_\eta$ has infinitely many rational points, so $F_\eta(\QQ)\cap (\mathcal{S}\times \mathcal{S})$ is infinite by \cref{prop:bijection}.
\end{proof}
This allows us to prove that every rational number can be written as a sum of three elements of $\mathcal{S}'$ in infinitely many ways; in other words, for any rational $t>0$, there are infinitely many ways to cut a $1\times t$ rectangle into three rectangles, each of which has rational distances between every pair of vertices. We also prove that every rational number can be written as a product of three elements of $\mathcal{S}'$ in infinitely many ways.
\begin{proof}[Proof of \cref{prop:threesum}]
	The equation $x_1+2x_2-t=0$ defines $F_\eta$ for $\eta=\begin{psmallmatrix}
		0&1\\2&-t
	\end{psmallmatrix}$. Now $\mathcal{H}_\eta$ is isomorphic to $\mathcal{H}_{\eta'}$ for $\eta'=\begin{psmallmatrix}
		1&-t/2\\0&-1/2
	\end{psmallmatrix}$, and by \cref{thm:Qinfpoints}, $\mathcal{H}_\eta(\QQ)$ is infinite for all $t\neq 0$. Hence every nonzero $t\in\QQ$ can be written as $\alpha_1+\alpha_2+\alpha_2$ for infinitely many pairs $(\alpha_1,\alpha_2)\in\mathcal{S}_\QQ^2$. 	The case $t=0$ follows from \cref{prop:pythslopesum} because $\mathcal{S}$ is closed under negation.

	Next we will show that for any $t\in\QQ^\times$, there exists $u\in\QQ-\{0,\pm 1\}$ such that when $s=-t\left(\frac{2u}{1-u^2}\right)$, the polynomial $x_1x_2+s$ has infinitely many solutions with $x_1,x_2\in\mathcal{S}'$. Each of these solutions can then be multiplied by $\frac{1-u^2}{2u}\in\mathcal{S}'$ to exhibit $t$ as a product of three elements of $\mathcal{S}'$.
	
	Let $\eta=\begin{psmallmatrix}
		1&0\\0&s
	\end{psmallmatrix}$. We consider the elliptic curve
	\[E_\eta:y^2=x(x+1)(x+s^2)=x(x+1)\left(x+t^2\left(\frac{2u}{1-u^2}\right)^2\right).\]
	If we set $u=t^2+2$, then the elliptic curve
	\[y^2=x(x+1)\left(x+t^2\left(\frac{2(t^2+2)}{1-(t^2+2)^2}\right)^2\right)\]
	over $\QQ$ has a point
	\[\left(\frac{t^2 (t^2+1)^2 (t^2+2)}{(t^2+3)^2},\,\frac{t^2 (t^2+2) (t^8+4t^6+6t^4+8t^2+9)}{(t^2+3)^3}\right),\]
	which has infinite order when $t\neq 0,\pm 1$. So for all $t\in\QQ-\{0,\pm 1\}$, $x_1x_2-t\left(\frac{2u}{1-u^2}\right)=0$ has infinitely many solutions $x_1,x_2\in\mathcal{S}'$, so $t$ can be written as a product of three elements of $\mathcal{S}'$ in infinitely many different ways.
	
	We finally must handle $t=\pm 1$. In this case, we can set $u=\frac{5}{6}$. For $\eta=\begin{psmallmatrix}
		1&0\\0&\pm\frac{60}{11}
	\end{psmallmatrix}$ we have the elliptic curve
	\[E_\eta:y^2=x^3+\left(1+\left(\frac{60}{11}\right)^2\right)x^2+\left(\frac{60}{11}\right)^2x,\]
	which has a non-torsion point $(-\frac{12}{11},\frac{204}{121})$ (in fact $E_\eta(\QQ)$ has rank $2$). Thus there are infinitely many Pythagorean solutions of $x_1x_2\mp \frac{60}{11}=0$, allowing us to write $\pm 1$ as a product $\alpha_1\alpha_2\frac{11}{60}$ of three elements of $\mathcal{S}'$ in infinitely many ways.
\end{proof}

\begin{rmk}
	The substitution $s=t^2+2$ was found essentially by trial and error, guided by inspiration from a MathOverflow answer by Siksek~\cite{siksek} describing how to find a positive rank subfamily of the family $y^2=x(x+1)(x+(\frac{1-s}{s})^2)$, and from Naskr\k{e}cki~\cite{naskrecki} who used a similar method to find a positive rank subfamily of the curve $y^2=x(x-1)\left(x-\left(\frac{2s}{1-s^2}\right)^2\right)$.
	
	For the $t=1$ case, the existence of a solution to $\alpha_1\alpha_2\alpha_3=1$ is equivalent to the existence of a body cuboid \cref{sec:cuboid}. The existence of a body cuboid with edge lengths $(240,117,44)$ leads to the choice of $u$.
\end{rmk}

By \cref{prop:pythslopesum}, every element of $\mathcal{S}'$ can be written as a sum of two elements of $\mathcal{S}'$ in infinitely many ways, but we have no comparable result for products. 
A natural question then is to determine which rational numbers $t$ can be written as a product of two elements of $\mathcal{S}'$ in infinitely many ways. This line of inquiry is explored in more depth in \cite{loveroot}.

\printbibliography

\end{document}